\documentclass[11pt,leqno]{amsart}

\usepackage[latin1]{inputenc}
\usepackage{amsmath}
\usepackage{amsfonts}
\usepackage{amssymb}
\usepackage{amsthm}
\usepackage[english]{babel}
\usepackage[T1]{fontenc}
\usepackage{graphicx}

\usepackage[pdftex]{hyperref}

\theoremstyle{plain}

\newtheorem{Thm}{Theorem}[section]
\newtheorem{Lem}[Thm]{Lemma}

\newcommand{\End}{\text{End}\,}

\begin{document}

\title{Positive laws on generators in powerful pro-$p$ groups}
\author{Cristina Acciarri}

\address{\textnormal{Cristina Acciarri,}\\ Dipartimento di Matematica Pura ed Applicata\\
Universit\`a degli Studi  dell'Aquila\\
I-67010 Coppito, L'Aquila (Italy)}
\email{acciarricristina@yahoo.it}

\author{Gustavo A. Fern\'andez-Alcober}
\address{ \textnormal{Gustavo A. Fern\'andez-Alcober,} Matematika Saila\\ Euskal Herriko Unibertsitatea
\\ 48080 Bilbao (Spain)}
\email{gustavo.fernandez@ehu.es}



\begin{abstract}
If $G$ is a finitely generated powerful pro-$p$ group satisfying a certain law $v\equiv 1$, and if $G$
can be generated by a normal subset $T$ of finite width which satisfies a positive law, we prove that
$G$ is nilpotent.
Furthermore, the nilpotency class of $G$ can be bounded in terms of the prime $p$, the number of
generators of $G$, the law $v\equiv 1$, the width of $T$, and the degree of the positive law.
The main interest of this result is the application to verbal subgroups: if $G$ is a $p$-adic analytic
pro-$p$ group in which all values of a word $w$ satisfy  positive law, and if the verbal subgroup
$w(G)$ is powerful, then $w(G)$ is nilpotent.
\end{abstract}

\maketitle

\section{Introduction}

If $\alpha$ and $\beta$ are two group words, we say that a group $G$ satisfies the
\emph{law} $\alpha\equiv\beta$ if every substitution of elements of $G$ by the variables
gives the same value for $\alpha$ and for $\beta$.
If the words $\alpha$ and $\beta$ are positive, i.e.\ if they do not involve any
inverses of the variables, then we say that $\alpha\equiv\beta$ is a
\emph{positive law}.
We can similarly speak about a law holding on a subset $T$ of $G$, if we only
substitute elements of $T$ by the variables.
Groups satisfying a positive law have received special attention in the past decade.
The main result is due to Burns and Medvedev, who proved in Ref.\ \cite{Burns}
that a locally graded group $G$ satisfies a positive law if and only if $G$
is nilpotent-by-(locally finite of finite exponent).
This applies in particular to residually finite groups.

A similar kind of problem has been considered by Shumyatsky and the second author
in Ref.~\cite{Fernandez-Alcober}.
If $G$ is a finitely generated group and $T$ is a set of generators satisfying a
positive law, they ask whether the whole of $G$ will also satisfy a (possibly different)
positive law, provided that $T$ is sufficiently large in some sense.
In this direction, they obtain a positive answer if $T$ is a normal subset of $G$
which is closed under taking commutators of its elements (\emph{commutator-closed}
for short), under the assumption that $G$ satisfies an arbitrary law and is residually-$p$
for some prime $p$.
More precisely, the result is proved for all primes outside a finite set $P(n)$
depending only on the \emph{degree} $n$ of the law (that is, the maximum of the
lengths of $\alpha$ and $\beta$).

The result in the previous paragraph can be applied to verbal subgroups $w(G)$
in a group $G$, where $T$ is considered to be the set $G_w$ of all values of the
word $w$ in $G$.
Note that $G_w$ is always a normal subset.
Among other results, Shumyatsky and the second author prove that, if $G$ is a
$p$-adic analytic pro-$p$ group and $p\not\in P(n)$ then, for every word $w$ such
that $G_w$ is commutator-closed, a positive law on $G_w$ implies a positive law on
the whole of $w(G)$.
Now two questions naturally arise: (i) can we get rid of the restriction
$p\not\in P(n)$?; (ii) can we get rid of the condition that $G_w$ should be
commutator-closed?
If we can give a positive answer to both these questions, then the result will hold
in $p$-adic analytic pro-$p$ groups for all primes and for all words.

If $G$ is a $p$-adic analytic pro-$p$ group, Jaikin-Zapirain has proved
(see Theorem 1.3 of Ref.~\cite{Jaikin}) that the set $G_w$ has finite width for every
word $w$, and then, by Proposition 4.1.2 of Ref.~\cite{Segal}, the verbal subgroup $w(G)$
is closed in $G$.
(See Section \ref{unipotent actions} for the definition of width.)
Thus $w(G)$ is again a $p$-adic analytic pro-$p$ group and, according to
Interlude A of Ref.\ \cite{DDMS}, it contains a powerful subgroup of finite index.
One of the main results of this paper is the solution of the problem raised in the last
paragraph in the case when $w(G)$ itself is powerful.

\begin{Thm}
\label{for verbal}
Let $G$ be a $p$-adic analytic pro-$p$ group, and let $w$ be any word.
If all values of $w$ in $G$ satisfy a positive law and the verbal subgroup $w(G)$ is powerful,
then $w(G)$ is nilpotent.
\end{Thm}

Observe that the conclusion in the previous theorem that $w(G)$ is nilpotent is actually stronger
than $w(G)$ satisfying a positive law.

Following the approach of Ref.~\cite{Fernandez-Alcober}, we obtain Theorem \ref{for verbal} from
a more general result not involving directly word values.
In this case, we work with $G$ a finitely generated powerful pro-$p$ group for an arbitrary
prime $p$, and the set of generators $T$ has to be normal and of finite width, but not necessarily
commutator-closed.
Recall that, as mentioned above, if $G$ is a $p$-adic analytic pro-$p$ group, then $G_w$ has finite
width for every word $w$.

\begin{Thm}
\label{for general}
Let $G$ be a powerful $d$-generator pro-$p$ group which satisfies a certain law $v\equiv 1$.
Suppose that $G$ can be generated by a normal subset $T$ of width $m$ that satisfies a
positive law of degree $n$.
Then $G$ is nilpotent of bounded class.
\end{Thm}

Here, and in the remainder of the paper, when we say that a certain invariant of a group
is bounded, we mean that it is bounded above by a function of the parameters appearing
in the statement of the corresponding result.
Thus, in Theorem \ref{for general}, the nilpotency class of $G$ is bounded in terms of
the prime $p$, the number $d$ of generators of $G$, the law $v\equiv 1$, the width $m$ of $T$,
and the degree $n$ of the positive law.
If we want to make explicit the set $S$ of parameters in terms of which a certain
quantity is bounded, then we will use the expression `$S$-bounded'.

We want to remark that, contrary to what happens in Theorem \ref{for general}, we cannot
guarantee that the nilpotency class of $w(G)$ is bounded in Theorem \ref{for verbal}.
The reason is that we are using the above-mentioned result of Jaikin-Zapirain, which
provides the finite width of $G_w$, but not bounded width for that set.

\section{The action on abelian normal sections}
\label{unipotent actions}

Our first step is to translate the positive law on the normal generating set $T$ into
a condition about the action of the elements of $T$ on the abelian normal sections of $G$.
More precisely, we have the following consequence of Lemma 2.1 in Ref.~\cite{Fernandez-Alcober}.
(Let $f(X)$ be the product of the polynomials $f_1(X)$ and $f_{-1}(X)$ in the statement of that
lemma.)

\begin{Lem}
\label{action on abelian}
Let $T$ be a normal subset of a group $G$, and assume that $T$ satisfies a positive law of
degree $n$.
Then there exists a monic polynomial $f(X)\in\mathbb{Z}[X]$ of degree $2n$, depending only on the
given positive law, which satisfies the following property:
if $A$ is an abelian normal section of $G$, then $f(t)$, viewed as an endomorphism
of $A$, is trivial for every $t\in T\cup T^{-1}$.
\end{Lem}

If $T$ is a subset of a group $G$, we say that $T$ has \emph{finite width}
if there exists a positive integer $m$ such that every element of the
subgroup $\langle T \rangle$ can be expressed as a product of no more than
$m$ elements of $T\cup T^{-1}$.
The smallest possible value of $m$ is then called the \emph{width} of $T$.

In our next theorem, we show how some properties of the generating set $T$ of $G$
are hereditary for the natural generating set of $\gamma_k(G)$ which can be
constructed from $T$.
For simplicity, if $A=K/L$ is a normal section of a group $G$, we say that two
elements $g,h\in G$ commute modulo $A$ if $gL$ and $hL$ commute modulo $A$
(or, equivalently, if $g$ and $h$ commute modulo $K$).

\begin{Thm}
\label{Tk of finite width}
Let $G$ be a $d$-generator finite $p$-group, and let $T$ be a normal
generating set of $G$.
Then
\[
T_k = \{ [t_1,\ldots,t_k] \mid t_i\in T \}
\]
is a normal generating set of $\gamma_k(G)$, and furthermore:
\begin{itemize}
\item[(i)]
If $T$ has finite width $m$, then the width of $T_k$ is at most $md^{(k-1)}$.
\item[(ii)]
If $T$ satisfies a positive law of degree $n$, then there exists a monic
polynomial $h(X)\in\mathbb{Z}[X]$ of $n$-bounded degree such that $h(t_k)$ annihilates
$\gamma_{k+1}(G)/\gamma_{k+1}(G)'$ for every $t_k\in T_k\cup T_k^{-1}$.
\end{itemize}
\end{Thm}

\begin{proof}
Of course, $T_k$ is a normal subset of $G$, and the proof that $T_k$ generates
$\gamma_k(G)$ is routine.

(i)
We argue by induction on $k$.
The result is obvious for $k=1$, so we assume next that $k\ge 2$.
By the Burnside Basis Theorem, we can choose $t_1,\ldots,t_d\in T$ such that
$G=\langle t_1,\ldots,t_d \rangle$.
If $y$ is an arbitrary element of $\gamma_k(G)$, we can write
\begin{equation}
\label{segal}
y = [g_1,t_1] \ldots [g_d,t_d],
\quad
\text{for some $g_i\in\gamma_{k-1}(G)$},
\end{equation}
by using Proposition 1.2.7 of  Ref.~\cite{Segal}.
Now, if $g$ is an arbitrary element of $\gamma_{k-1}(G)$, then by the induction
hypothesis, we have $g=u_1\ldots u_s$ for some $u_i\in T_{k-1}\cup T_{k-1}^{-1}$, where
$s\le md^{k-2}$.
Then, for every $t\in T$, we have
\[
[g,t] = [u_1,t]^{u_2\ldots u_s} \ldots [u_{s-1},t]^{u_s} [u_s,t].
\]
If $u_i\in T_{k-1}$, then $[u_i,t]\in T_k$; on the other hand, if $u_i\in T_{k-1}^{-1}$ then
\[
[u_i,t] = \left( [u_i^{-1},t]^{u_i} \right)^{-1}
\]
is an element of $T_k^{-1}$.
Thus $[g,t]$ is a product of at most $s$ elements of $T_k\cup T_k^{-1}$, and it follows from
(\ref{segal}) that $y$ is a product of no more than $ds$ elements of $T_k\cup T_k^{-1}$.
This completes the proof of (i).

(ii)
Set $A=\gamma_{k+1}(G)/\gamma_{k+1}(G)'$.
By Lemma \ref{action on abelian}, there exists a monic polynomial
$f(X)\in\mathbb{Z}[X]$ of degree $2n$ such that $f(t)$ annihilates $A$ for
every $t\in T\cup T^{-1}$.

Let $I$ be the ideal of $\mathbb{Z}[X_1,X_2]$ generated by $f(X_1)$ and
$f(X_2)$.
Since $f$ is monic, the quotient ring $R=\mathbb{Z}[X_1,X_2]/I$ is a
finitely generated $\mathbb{Z}$-module, generated by the images of the
monomials $X_1^iX_2^j$ with $0\le i,j\le 2n-1$.
By Theorem 5.3 in Chapter VIII of  Ref.~\cite{Hungerford}, $R$ is integral over
$\mathbb{Z}$.
In particular, there exists a monic polynomial $h(X)\in \mathbb{Z}[X]$
such that $h(X_1X_2)\in I$.
Also, by examining the proof of that result in Ref.~\cite{Hungerford}, it is
clear that the degree of $h(X)$ is at most $(2n)^2$.

Now, let $[u,t]$ be an arbitrary element of $T_k$, where $u\in T_{k-1}$
and $t\in T$.
Since $(t^u)^{-1}$ and $t$ commute modulo $A$, these elements define commuting
endomorphisms of $A$, and hence we can define a ring homomorphism
\[
\begin{matrix}
\varphi & : & \mathbb{Z}[X_1,X_2] & \longrightarrow & \End(A)
\\
& & X_1 & \longmapsto & (t^u)^{-1}
\\
& & X_2 & \longmapsto & t.
\end{matrix}
\]
Since $f((t^u)^{-1})$ and $f(t)$ are both the null endomorphism of $A$,
it follows that $f(X_1)$ and $f(X_2)$ are contained in the kernel of $\varphi$,
and so the same holds for the ideal $I$.
Hence $h(X_1X_2)\in\ker\varphi$, which means that $h([u,t])$ is the null
endomorphism of $A$.

We can similarly prove that $h([t,u])=0$ in $\End(A)$, by defining
$\psi:\mathbb{Z}[X_1,X_2]\longrightarrow \End(A)$ via the assignments $X_1\mapsto t^{-1}$
and $X_2\mapsto t^u$.
Thus $h(t_k)$ annihilates $A$ for every $t_k\in T_k\cup T_k^{-1}$.
\end{proof}

Finally, for a certain $k$, we are able to get an Engel action of all $k$-th
powers of the elements of $G$ on some abelian normal sections of $G$.

\begin{Thm}
\label{unipotent action of a power}
Let $G$ be a finite $p$-group generated by a normal subset $T$ which
has width $m$.
Suppose that $A$ is an abelian normal section of $G$ such that the elements of $T$
commute pairwise modulo $A$, and that for some monic polynomial $f(X)\in\mathbb{Z}[X]$,
$f(t)$ annihilates $A$ for all $t\in T\cup T^{-1}$.
Then:
\begin{itemize}
\item[(i)]
There exists an $\{m,f\}$-bounded integer $r$ such that
$[A,_r g]\le A^p$ for every $g\in G$.
\item[(ii)]
There exist $\{m,f\}$-bounded integers $n$ and $k$ such that
$[A,_n g^k]=1$ for every $g\in G$.
\end{itemize}
\end{Thm}

\begin{proof}
The first part of the proof is similar to the proof of (ii) in the
last theorem.
Let us write $n$ for the degree of $f(X)$.
Consider the quotient ring $R=\mathbb{Z}[X_1,\ldots,X_m]/I$, where
$I$ is the ideal generated by the polynomials $f(X_1),\ldots,f(X_m)$.
Then $R$ is integral over $\mathbb{Z}$, and there exists a monic
polynomial $h(X)\in\mathbb{Z}[X]$ of degree at most $n^m$ such that
$h(X_1\ldots X_m)\in I$.
Now let $g$ be an arbitrary element of $G$.
Since $T$ generates $G$ and has width $m$, we can write $g=t_1\ldots t_m$
for some $t_i\in T\cup T^{-1}$.
The map $X_1\mapsto t_1,\ldots,X_m\mapsto t_m$ extends to a ring homomorphism
$\varphi:\mathbb{Z}[X_1,\ldots,X_m]\longrightarrow \End(A)$, since
the elements of $T$ commute pairwise modulo $A$.
Since $f(t_1)=\cdots=f(t_m)=0$, it follows that $I\subseteq \ker\varphi$.
Consequently,
 $$h(g)=h(t_1\ldots t_m)=\varphi(h(X_1\ldots X_m))=0.$$
Thus we have found a monic polynomial $h(X)\in\mathbb{Z}[X]$ such that
$h(g)$ annihilates $A$ for all $g\in G$.
Note that the polynomial $h(X)$ only depends on $f(X)$ and $m$, but not on the
particular element $g$ or on the section $A$.

(i)
Since $G$ is a finite $p$-group, we have $[A,_c G]=1$ for some $c$.
Let $(X-1)^r$ be the greatest common divisor of $(X-1)^c$ and $h(X)$, when
these polynomials are considered in $\mathbb{F}_p[X]$.
Since $r\le \deg h$, it follows that $r$ is $\{m,f\}$-bounded.
By B\'ezout's identity, we can write
\[
(X-1)^r = p(X)(X-1)^c + q(X)h(X),
\]
for some $p(X),q(X)\in \mathbb{F}_p[X]$.
If we consider an element $g\in G$, and substitute $g$ for $X$ in the previous
expression, then, as endomorphisms of the $\mathbb{F}_p$-vector space $A/A^p$,
we get $(g-1)^r=0$.
This means that $[A,_r g]\le A^p$, as desired.

(ii)
Let $J$ be the ideal of $\mathbb{Z}[X]$ generated by all polynomials
$h(X^i)$ with $i\ge 1$.
Then, if $j(X)\in J$, it follows that $j(g)=0$ for every $g\in G$.
By Lemma 3.3 of  Ref.~\cite{Semple}, there exist positive integers $q,k$
and $\ell$ such that
\[
qX^{\ell}(X^k-1)^{\ell}\in J,
\]
where $q,k,\ell$ depend only on $h(X)$, so only on $f(X)$ and $m$.
Then
\[
A^{qg^{\ell}(g^k-1)^{\ell}}=1,
\quad
\text{for every $g\in G$.}
\]

If $p^s$ is the largest power of $p$ which divides $q$, then $A^q=A^{p^s}$,
since $A$ is a finite $p$-group.
Also, we have $A^g=A$.
Hence
\[
A^{p^s(g^k-1)^{\ell}}=1
\]
or, what is the same,
\begin{equation}
\label{engel-like}
[A^{p^s},g^k,\overset{\ell}{\ldots},g^k]=1
\end{equation}
for every $g\in G$.

Now, it follows from part (i) that
\[
[A^{p^i},_r g]\le A^{p^{i+1}},
\quad
\text{for every $i\ge 0$, and for every $g\in G$.}
\]
This, together with (\ref{engel-like}), shows that
\[
[A,_n g^k]=1,
\quad
\text{for all $g\in G$,}
\]
where $n=sr+\ell$.
\end{proof}

\section{Proof of the Main Theorems}

We will begin by proving Theorem \ref{for general}.
In order to show that the powerful pro-$p$ group $G$ is nilpotent, we
will rely on the following two lemmas.
The first one is a classical result of Philip Hall
(see, for example, Theorem 3.26 of Ref.~\cite{Khukhro}), and
the other one says that for a finitely generated powerful pro-$p$ group
`nilpotent-by-finite' is the same as `nilpotent'.

\begin{Lem}
\label{Hall}
Let $G$ be a group, and let $N$ be a normal subgroup of $G$.
If $N$ is nilpotent of class $k$ and $G/N'$ is nilpotent of class $c$,
then $G$ is nilpotent of $\{k,c\}$-bounded class.
\end{Lem}

\begin{Lem}
\label{nilpotent-by-finite powerful}
Let $G$ be a finitely generated powerful pro-$p$ group.
If $G$ has a normal subgroup $N$ of finite index which is nilpotent of
class $c$, then $G$ itself is nilpotent of $\{c,e\}$-bounded class,
where $e$ is the exponent of $G/N$.
\end{Lem}

\begin{proof}
We prove the result for $p>2$.
For $p=2$, the same proof applies with some little changes.

It follows from the hypotheses that $G^e$ is nilpotent of class at
most $c$.
By Proposition 3.2 and Corollary 3.5 in Ref.\ \cite{FGJ}, we get
\begin{equation}
\label{long commutator}
[G^{e^{c+1}},G,\overset{c}{\ldots},G]
=
[G,\overset{c+1}{\ldots},G]^{e^{c+1}}
=
[G^e,\overset{c+1}{\ldots},G^e]
=
1.
\end{equation}
On the other hand, since $G$ is powerful, we have
$\gamma_{i+1}(G)\le G^{p^i}$ for all $i\ge 1$.
As a consequence, for some $\{c,e\}$-bounded integer $k$
we have $\gamma_{k+1}(G)\le G^{e^{c+1}}$.
This, together with (\ref{long commutator}), shows that
$G$ is nilpotent of class at most $k+c$, and we are done.
\end{proof}

Note that we could have written the previous lemma under the
apparently weaker assumption that the exponent of $G/N$ is finite,
rather than $N$ being of finite index in $G$.
However, if $G$ is a finitely generated powerful pro-$p$ group, these
two conditions are equivalent: if $\exp G/N=p^k$, then $G^{p^k}$ is
contained in $N$, and then by Theorem 3.6 of Ref.~\cite{DDMS}, we have
$|G:N|\le |G:G^{p^k}|\le p^{kd}$, where $d$ is the minimum number of
generators of $G$ as a topological group.
(In fact, the assumption that $G$ should be powerful is not necessary
for this equivalence, since $|G:G^{p^k}|$ is finite for every finitely
generated pro-$p$ group.
But this is a much deeper result, which needs Zelmanov's positive
solution of the Restricted Burnside Problem.)

\vspace{10pt}

We also need the following result of Black (see
Corollary 2 in Ref.~\cite{Black}).

\begin{Thm}
\label{black}
Let $G$ be a finite group of rank $r$ satisfying a law $v\equiv 1$.
Then, there exists an $\{r,v\}$-bounded number $k$ such that
$\gamma_k((G^{k!})')=1$.
In particular, if $G$ is soluble, then the derived length of $G$
is $\{r,v\}$-bounded.
\end{Thm}

Note that the positive solution to the Restricted Burnside Problem is needed
for the conclusion in the soluble case: thus we know that the quotient $G/G^{k!}$
has bounded order, and so also bounded derived length.

We can now proceed to the proof of Theorem \ref{for general}.

\begin{proof}[Proof of Theorem \ref{for general}]
Suppose that the result is known for $G$ a finite $p$-group, so that all finite $p$-groups
satisfying the conditions of the theorem have nilpotency class at most $c$, for some bounded
number $c$.
Now if $G$ is a pro-$p$ group as in the statement of the theorem, and $N$ is an arbitrary open
normal subgroup of $G$, it follows that $\gamma_{c+1}(G)\le N$.
Thus necessarily $\gamma_{c+1}(G)=1$ and the result is valid also for pro-$p$ groups.

Hence we may assume that $G$ is a $d$-generator finite powerful $p$-group.
By Theorem 11.18 of Ref.~\cite{Khukhro}, it follows that $G$ has rank $d$, i.e.\ that
every subgroup of $G$ can be generated by $d$ elements.
Since $G$ satisfies the law $v\equiv 1$, by Theorem \ref{black} we have $G^{(s)}=1$ for some
bounded number $s$.
We argue by induction on $s$.

If $s\le 2$, i.e.\ if $G$ is metabelian, then the elements of $T$ commute pairwise
modulo $G'$.
Choose generators $g_1,\ldots,g_d$ of $G$.
By Lemma \ref{action on abelian} and Theorem \ref{unipotent action of a power}, since $T$
satisfies a positive law, we know that there exist bounded numbers $n$ and $k$ such that
$[G',_n g_i^k]=1$ for all $i=1,\ldots,d$.
As a consequence, the subgroups $\langle g_i^k,G' \rangle$ have bounded nilpotency class.
Thus $G^kG'=\langle g_1^k,\ldots,g_d^k,G' \rangle$ is the product of $d$ normal subgroups
of bounded class, and so has bounded class itself.
Since $|G:G^kG'|\le k^d$, it follows from Lemma \ref{nilpotent-by-finite powerful} that
$G$ has bounded nilpotency class.
This concludes the proof in the metabelian case.

Assume now that $s\ge 3$.
We claim that the nilpotency class of $G/\gamma_{k+1}(G)'$ is bounded for all
$k\ge 1$ (here, we must also take $k$ into account for the bound).
The result is true for $k=1$, according to the last paragraph.
Now we argue by induction on $k$.
By Theorem \ref{Tk of finite width}, $T_k$ is a normal set of generators of $\gamma_k(G)$
of bounded width.
Also, the elements of $T_k$ commute pairwise modulo $\gamma_{k+1}(G)$.
On the other hand, by (ii) of Theorem \ref{Tk of finite width}, there exists a monic
polynomial $h(X)\in\mathbb{Z}[X]$ such that $h(t_k)$ annihilates the abelian normal section
$A=\gamma_{k+1}(G)/\gamma_{k+1}(G)'$ for every $t_k\in T_k$.
Thus we may argue as in the metabelian case above with the group $Q=\gamma_k(G)/\gamma_{k+1}(G)'$
and deduce that $Q$ has bounded nilpotency class.
Since $G/\gamma_k(G)'$ has also bounded class by the induction hypothesis, the claim follows
from Lemma \ref{Hall}.

Now that the claim is proved, the result easily follows.
Indeed, since $G/G^{(s-1)}$ has bounded class by induction, there exists a bounded integer $\ell$
such that $\gamma_{\ell+1}(G)\le G^{(s-1)}$.
Hence $\gamma_{\ell+1}(G)'=1$, and $G$ has bounded class by the previous claim.
\end{proof}

Now Theorem \ref{for verbal} follows readily from Theorem \ref{for general}.

\begin{proof}[Proof of Theorem \ref{for verbal}]
As already mentioned, the set $G_w$ of all values of $w$ in $G$ is a
normal subset of $G$, and in particular of $w(G)$.
Also, by Theorem 1.3 of Ref.\ \cite{Jaikin}, $G_w$ has finite width,
say $m$.

Let $\alpha\equiv\beta$ be the positive law satisfied by the set $G_w$,
and suppose that the number of variables used in the law $\alpha\equiv\beta$
and in the word $w$ is $k$ and $\ell$, respectively.
Now, consider $kl$ arbitrary elements $g_1,\ldots,g_{kl}$ of $G$.
Since the $k$ elements $w(g_1,\ldots,g_{\ell}),\ldots,
w(g_{(k-1)\ell+1},\ldots,g_{k\ell})$ satisfy the law $\alpha\equiv\beta$,
it follows that
\begin{multline*}
\alpha(w(g_1,\ldots,g_{\ell}),\ldots, w(g_{(k-1)\ell+1},\ldots,g_{k\ell}))
=
\\
\beta(w(g_1,\ldots,g_{\ell}),\ldots, w(g_{(k-1)\ell+1},\ldots,g_{k\ell})).
\end{multline*}
This means that the group $G$ satisfies a law $v\equiv 1$, where $v$ is a
word which depends only on $w$ and on the positive law $\alpha\equiv\beta$.
In particular, the law $v\equiv 1$ is also satisfied by $w(G)$.

Now, we can apply directly Theorem \ref{for general} to the group $w(G)$
and the generating set $G_w$, in order to conclude that $w(G)$ is nilpotent.
\end{proof}

\section*{Acknowledgments}
The authors are supported by the Spanish Government, grant MTM2008-06680-C02-02,
partly with FEDER funds, and by the Department of Education, Universities and Research
of the Basque Government, grants IT-252-07 and IT-460-10.
The second author is also supported by a grant of the University of L'Aquila.

\end{document}